\documentclass[a4paper,11pt,oneside]{article}
\usepackage{blindtext}
\usepackage{xcolor}
\usepackage[latin1]{inputenc}
\usepackage[russian,english]{babel}
\usepackage{amsmath,amssymb}
\usepackage{textcomp}
\usepackage{bbm}
\usepackage[T2A,T1]{fontenc}
\usepackage{amsthm}
\usepackage[all,cmtip]{xy}
\usepackage{mathtools}
\usepackage{tikz}
\usepackage{tikz-cd}
\usetikzlibrary{matrix}
\usepackage{tocbibind}
\usepackage{graphicx}
\usepackage{wrapfig}
\usepackage{mathrsfs}
\usepackage{fancyhdr}
\usepackage[colorinlistoftodos]{todonotes}
\usetikzlibrary{arrows,decorations.pathmorphing}
\title{}
\author{}
\date{}
\usetikzlibrary{calc,trees,positioning,arrows,chains,shapes.geometric,%
	decorations.pathreplacing,decorations.pathmorphing,shapes,%
	matrix,shapes.symbols}
\usepackage{titlesec}
\usepackage{graphicx}

\titleformat{\chapter}[display]
{\bfseries\huge}
{\filcenter\MakeUppercase{\chaptertitlename} \Huge\thechapter}
{1ex}
{\titlerule\vspace{1ex}\filcenter}
[\vspace{1ex}\titlerule]
\usepackage{scrextend}

\numberwithin{equation}{section}

\newtheorem{conj}{Conjecture}
\newtheorem{thm}{Theorem}[section]

\newtheorem{prop}[thm]{Proposition}

\newtheorem{coroll}[thm]{Corollary}

\newtheorem{lemma}[thm]{Lemma}

\theoremstyle{remark}
\newtheorem{osserv}{Remark}[section]

\theoremstyle{definition}

\newcommand{\Z}{\mathbb{Z}}
\newcommand{\Co}{\mathbb{C}}
\newcommand{\R}{\mathbb{R}}

\newcommand{\Q}{\mathbb{Q}}
\newcommand{\p}{\mathbb{P}}
\newcommand{\Qal}{\overline{\Q}}

\theoremstyle{plain} 
\newcommand{\thistheoremname}{}
\newtheorem{genericthm}[thm]{\thistheoremname}
\newenvironment{namedthm}[1]
{\renewcommand{\thistheoremname}{#1}%
	\begin{genericthm}}
	{\end{genericthm}}
\newtheorem*{genericthm*}{\thistheoremname}
\newenvironment{namedthm*}[1]
{\renewcommand{\thistheoremname}{#1}%
	\begin{genericthm*}}
	{\end{genericthm*}}

\DeclareMathOperator{\ai}{a}
\DeclareMathOperator{\bi}{b}
\DeclareMathOperator{\ci}{c}

\DeclareMathOperator{\Frob}{Frob}

\DeclareMathOperator{\paruno}{(}
\DeclareMathOperator{\pardue}{)}

\DeclareMathOperator{\h}{ht}

\makeatletter
\newcommand*{\math@version@bold}{bold}
\DeclareMathOperator\Sha{
	\textrm{%
		\usefont{T2A}{cmr}{\ifx\math@version\math@version@bold bx\else m\fi}{n}%
		\CYRSH
	}%
}
\makeatother

\newcommand\blfootnote[1]{%
	\begingroup
	\renewcommand\thefootnote{}\footnote{#1}%
	\addtocounter{footnote}{-1}%
	\endgroup
}

\begin{document}
	\begin{center}
		\textbf{ARITHMETIC STATISTICS OF FAMILIES OF\\INTEGER $S_n$-POLYNOMIALS AND APPLICATION TO CLASS GROUP TORSION}
	\end{center}
\begin{center}
	Ilaria Viglino
\end{center}
\begin{center}
	\textbf{Abstract}
	
\end{center}
\begin{addmargin}[2em]{2em}

\fontsize{10pt}{12pt}\selectfont
	We study the distributions of the splitting primes in certain families of number fields. The first and main example is the
	family $\mathscr{P}_{n,N}$ of polynomials $f\in\Z[X]$ monic of degree $n$ with height less or equal then $N$, and then let $N$ go to infinity.
 We prove an average version of the Chebotarev Density Theorem for this family. In particular, this gives a Central Limit Theorem for the number of primes with given splitting type in some ranges. As an application, we deduce some estimates for the $\ell$-torsion in the class groups and for the average of ramified primes.
\end{addmargin}
\normalsize
\section{Introduction}

Let $n\ge2$ and $N$ be positive integers. The set $\mathscr{P}_{n,N}$ consists of all monic polynomials $$f(X)=X^n+a_{n-1}X^{n-1}+\dots+a_0$$in $ \Z[X] $, of degree $ n $ and we view the coefficients $ a_0,\dots,a_{n-1} $ as independent, identically distributed random variables taking values uniformly in $ \{-N,\dots,N \} $ with $ N\rightarrow\infty $. 

We are interested to a subset of $\mathscr{P}_{n,N}$, namely$$\mathscr{P}_{n,N}^{0}=\{f\in \mathscr{P}_{n,N}:G_f\cong S_n\},$$where $ K_f $ is the splitting field of $ f $ inside a fixed algebraic closure $ \Qal $ of $ \Q $, and $G_f$ is the Galois group of $K_f/\Q$.

The Galois group $G_f\simeq S_n$, the symmetric group, occurs with probability one. This was first proved by van der Waerden in \cite{Wa}, by using Hilbert Irreducibility Theorem. Later, Gallagher gave another proof usign large sieve in \cite{Gal}. On the other hand, all $S_n$-extensions of $\Q$ arise in this way for some $f$.\\
For a polynomial $f\in\mathscr{P}_{n,N}$, it is well known that$$\lim_{N\rightarrow\infty}\p(\mbox{f irreducible})=1.$$The error term, given by Chow and Dietmann \cite{CD} is $ O_n(N^{-1}) $ for $ n>2 $. Van der Waerden \cite{Wa} proved that $$\frac{|\mathscr{P}_{n,N}^{0}|}{|\mathscr{P}_{n,N}|}\underset{N\rightarrow+\infty}{\longrightarrow}1$$with explicit error term $ O_n(N^{n-\frac{1}{6(n-2)\log\log N}}) $. The error term has improved in \cite{Gal} using large sieve to $ O_n(N^{-1/2}\log N) $, and more recently by Dietmann \cite{Di} using resolvent polynomials to $O_n(N^{-2+\sqrt{2}+\varepsilon})$. The best estimate can be found in \cite{Bh1}, who proved the following result, conjectured by van der Waerden.

\begin{thm}[Bhargava]
	If $n\ge5$, one has, $$|\mathscr{P}_{n,N}^{0}(\Q)|=(2N)^n+O_n(N^{n-1}),$$as $ N\rightarrow\infty $.
\end{thm}The cubic and quartic cases of van der Waerden's conjecture were proved by Chow and Dietmann in \cite{CD}.

\blfootnote{This paper is an extract from my Ph.D. thesis, in which it is treated a more general case, as mentioned in Section 6.3.}

\subsection{Main results of this paper}

The main theorem that we state here is proven in a more general version in Section 4. It is a Central Limit Theorem. The limit below shows how $\pi_f(x)$ fluctuates about the central value $\pi(x)/n!$, the one expected by the Chebotarev Density Theorem. Let $f\in \mathscr{P}_{n,N}^0$ and let $\pi_f(x)$ the number of primes less or equal than $x$ splitting completely in $K_f/\Q$. Here $\pi(x)$ denotes the number of rational primes less or equal than $x$. \begin{thm}\label{theorem1}
	
	For $x\le N^{1/\log\log N}$ and for any real numbers $a<b$ one has$$\frac{1}{|\mathscr{P}_{n,N}^0|}\Big|\Big\{f\in\mathscr{P}_{n,N}^0:a\le\Big(\frac{1}{n!}-\frac{1}{n!^2}\Big)^{-1/2}\Big(\frac{\pi_f(x)-\frac{1}{n!}\pi(x)}{\pi(x)^{1/2}}\Big)\le b\Big\}\Big|$$ $$\underset{N\rightarrow+\infty}{\longrightarrow}\frac{1}{\sqrt{2\pi}}\int_{a}^{b}e^{-t^2/2}dt.$$
\end{thm}

\begin{osserv}
	The above result holds when $x$ is small compared to $N$, in particular $x\ll_\varepsilon N^{\varepsilon}$ for every $\varepsilon>0$. The reason behind it will be clarified in the proof of Lemma \ref{l1} and Proposition \ref{p1}. Roughly speaking, we want to ensure that there are enough (in a suitable sense) $S_n$-polynomials congruent to a given polynomial modulo primes $p\le x$. It would be interesting, and it is still an open question, to know what is the right range of $x$ and $N$ for which Theorem \ref{theorem1} holds.
	
\end{osserv}

\subsubsection{First application}

The main application we discuss here concerns some bounds for the $\ell$-torsion part of the class number of the above number fields, which improve on the trivial bound provided by the order of the entire class group. The class number of a number field $K/\Q$, that we denote by $h_K$, is by definition the order of the ideal class group of its ring of integers, that is, the abelian quotient group of the the group of fractional ideals and its subgroup of principal ideals. It measures the "distance" of a ring of integers to be a unique factorization domain; namely, this turns out to be the case if and only if the class number is 1.\\

Let $f\in\mathscr{P}_{n,N}^0$; we denote by $\Q_f$ the field associated to $f$, that is, the the field generated over $\Q$ by a root of $f$. Let $D_{\Q_f}$ be the absolute value of the discriminant of $\Q_f/\Q$, and let  $d_f$ be the discriminant of the polynomial $f$.
\begin{coroll}
	For every positive integer $\ell$, for almost all $ f \in\mathscr{P}_{n,N}^0$ outside of a set of size $ o(N^n) $, the $\ell$-torsion part of the class group of $\Q_f/\Q$ is bounded by$$h_{\Q_f}[\ell]\ll_{n,\ell,\varepsilon}D_{\Q_f}^{\frac{1}{2}-\frac{1}{(2n-2)\log\log d_f}+\varepsilon},$$for every $\varepsilon>0$, as $ N\rightarrow+\infty $.
\end{coroll}
This will be proved in Section 5.2, in which we also state the crucial result due to Ellenberg and Venkatesh (\cite{EV}, Lemma 2.3) in order to show the above corollary.
\subsubsection{Second application}

Let $f$ be an $S_n$-polynomial and let $d_f$ be its discriminant. The relation between $d_f$ and the discriminant $D_{\Q_f}$ of its field is still an open problem in many cases, and leads to difficulties when counting ramified primes.

Call an irreducible monic integral polynomial $f$ \textit{essential} if the equality between the two discriminant holds. It is well known that this
implies that the ring of integers of the splitting field of $f$ is monogenic.

Our results can be applied to study this relation, and to bound on average the number of primes dividing the discriminant of the extension $\Q_f/\Q$.

\begin{coroll}For almost all $ f \in\mathscr{P}_{n,N}^0$ outside of a set of size $ o(N^n) $, the number of ramified primes in $\Q_f/\Q$ is$$\ll_{n}\log\log N,$$as $N\rightarrow+\infty$.
\end{coroll}
See Section 5.1, in which we will state a slightly stronger result.

\section{Notations and background}

\subsection{The big-$O$ and the Vinogradov's notation}
We will make frequent use of various symbols to compare
the asymptotic sizes of quantities. We write $f(x)\ll_a g(x)$ or $f(x) = O_a(g(x))$ to state that there
exists a constant $C=C(a) > 0$ depending on $a$, such that $|f(x)|\le C|g(x)|$ for all $x$ suffciently large. Similarly, we write $f(x)\gg g(x)$ if there exists a constant
$C > 0$ such that $|f(x)|\le C|g(x)|$ for all $x$ suffciently large. We write $f(x)\asymp g(x)$
if both $f(x)\ll g(x)$ and $f(x)\gg g(x)$ hold. Moreover, we state that $f(x)\sim g(x)$ if $f(x)/g(x)\rightarrow1$ as
$x\rightarrow+\infty$, and $f(x)=o(g(x))$ if $f(x)/g(x)\rightarrow0$ as $x\rightarrow+\infty$.

\subsection{The Frobenius automorphism}
For the sake of clarity, we recall that given a normal extension $K$ over $\Q$ of degree $s$, the \textit{Galois group} $G_K=G_{K/\Q}$ of $K$ over $\Q$ is defined to be the group of automorphisms of $K$ that fix $\Q$ pointwise. There is a natural embedding$$G_K\hookrightarrow S_s$$given by the action of the Galois group on the $s$ homomorphisms of $K$ into $\overline{\Q}$. In the following we'll identify $G_K$ with its image via the above morphism. If $p$ is an unramified prime in $K/\Q$, i.e. the inertia group for every prime $\mathfrak{p}$ of $K$ is trivial, there is a canonical isomorphism between the Galois group of the residue fields exension $(\mathcal{O}_K/\mathfrak{p})/\mathbb{F}_p$ and the decomposition group $D_{\mathfrak{p}}$ at $\mathfrak{p}$. Now, $G_{\mathcal{O}_K/\mathfrak{p}/\mathbb{F}_p}$ is cyclic with canonical generator the Frobenius at $p$. The corresponding to $\mathfrak{p}$ is $\Frob_{K/\Q,\mathfrak{p}|p}\in D_{\mathfrak{p}}$. It is the unique element of $G_K$ such that for all $\alpha\in\mathcal{O}_K$ we have$$\Frob_{K/\Q,\mathfrak{p}|p}(\alpha)\equiv \alpha^p\mod\mathfrak{p}.$$If we consider another prime over $p$, that is a conjugated one through an element of the Galois group, the new Frobenius is conjugated to the previous one via the same automorphism. Hence we denote by $\Frob_{K/\Q,p}$ the \textit{Frobenius element} at $p$, namely the conjugacy class of Frobenius automorphisms in $G_K$.

\subsubsection{The number of splitting primes}

Let $ f\in \mathscr{P}_{n,N}^{0} $. We say that $ r=(r_1,r_2,\dots,r_n) $ is the \textbf{splitting type} of $ f$ mod a prime $ p $ if $ f\mod p $ splits into distinct monic irreducible factors (so a square-free factorization), with $ r_1 $ linear factors, $ r_2 $ quadratic factors and so on. For the primes $ p $ that not divide the discriminant $ D_f $ of the extension $ K_f/\Q $, $ r $ corresponds to the cycle structure of the Frobenius element $ \Frob_{K_f/\Q,p}=:\Frob_{f,p} $ acting on the roots of $ f $. For each $ r $ we have $$\sum_{i=1}^{n}ir_i=n.$$Let $ \mathscr{C}_r $ be the conjugacy class in $ S_n $ of elements of cycle type $ r $; then the order of $ \mathscr{C}_r $ is $ n!\delta(r) $, where 

\begin{equation}
	\delta(r)=\prod_{i=1}^{n}\frac{1}{i^{r_i}r_i!} 
	\label{eq1}
\end{equation}
 In general, for a subgroup $G\subseteq S_n$, the elements in a conjugacy class in $G$ necessarily have the same cycle type, but the converse need not to be true. That is, the cycle type of a conjugacy class in $G$ need not determine in uniquely. This uniqueness property does hold for cycle types for the full symmetric group, which implies that the cycle type of an $S_n$-polynomial having a square-free factorization mod $p$ uniquely determines the Frobenius element for an $S_n$-number field obtained by adjoining one root of it.

If $p$ is a prime number not dividing the discriminant of $f$, then we can write $f$ modulo $p$ as a product of distinct irreducible factors
over the finite field $\mathbb{F}_p$. The degrees of these irreducible factors form the splitting
type of $f$ mod $p$; this is also a partition of $n$. Frobenius's theorem
asserts, roughly speaking, that the number of primes with a given
decomposition type is proportional to the number of permutations in $S_n$ with the
same cycle pattern.

\begin{thm}[Frobenius, 1880]
	The density of the set of primes $p$ for which
	$f$ has a given splitting type $r_1,\dots,r_n$ exists, and it is equal to $1/|G_f|$ times the number of $\sigma\in G_f$ with cycle pattern $r_1,\dots,r_n$.
\end{thm}

We want to compute the density (on average) of primes $p$ unramified in $K_f/\Q$ for which $f$ has a given splitting type modulo $p$. This is a special case of the Chebotarev density theorem, for which we
want an explicit asymptotic, with an effective error term, for "almost
all" polynomials in our family.

For $ x\ge1 $, let$$\pi_{f,r}(x)=\underset{f\tiny{\mbox{ splitting type }}r\tiny{\mbox{ mod }} p}{\sum_{p\le x}}1=\sum_{p\le x}\mathbbm{1}_{f,r}(p),$$where $$\mathbbm{1}_{\cdot,r}(p):\mathscr{P}_{n,N}^{0}\longrightarrow\{0,1\}$$takes value 1 if $f$ has splitting type $r$ modulo $p$, and 0 otherwise.
Define $\p_N$ to be the uniform probability measure on the set $ \mathscr{P}_{n,N}^{0} $, seen as a sublattice of $ [-N,N]^n $.

We abbreviate $ \mathbbm{1}_{f,r}(p) $ with $ \mathbbm{1}_{p}(f) $.
We denote by $\mathbb{E}_N$ and $ \sigma^{2}_N $ the expectation and the variance of a random variable on $ \mathscr{P}_{n,N}^{0} $, respectively.\\
In Proposition \ref{p1} below, we compute mean and variance of the random variable $ \mathbbm{1}_{p}$.

\section{Prime splitting densities}

For every splitting type $ r $, let
\begin{align*}
	\mathscr{X}_{n,r,p}=\Big\{&\Big( \prod_{i=1}^{r_1}g_i^{(1)}\Big) \dots\Big(  \prod_{i=1}^{r_n}g_i^{(n)}\Big):g_i^{(j)}\in\mathbb{F}_p[X]\ \mbox{irreducible, monic},\\
	&\deg(g_i^{(j)})=j,\ g_i^{(j)}\neq g_k^{(j)}\mbox{ if }i\neq k\Big\}.
\end{align*}
The following key fact is what we will use to estimate the error term in the asymptotic of the expectation $\mathbb{E}_N(\pi_{f,r}(x))$ of $\pi_{f,r}(x)$ and its powers. The fact that the primes $p$ are small enough with respect to $N$, allows to have "many" $S_n$-polynomials congruent to given polynomials mod $p$.
\begin{lemma}\label{l1} Let $N>2$ and 
	let $k\ge 1$, $p_1,\dots, p_k$ primes and $g_i\in 	\mathscr{X}_{n,r,p_i}$ for all $i=1,\dots,k$. Then if $p_i<N^{1/kn}$ for all $i=1,\dots,k,$ $$\p_N(f\in\mathscr{P}_{n,N}^{0}:f\equiv g_i\mod p_i\ \ \forall i=1,\dots,k)=\frac{1}{(p_1\dots p_k)^n}+O_{n,k}(N^{-1}).$$
	
\end{lemma}

\begin{proof}We prove the case $k=1$. An application of the Chinese Remainder Theorem leads to the result for $k>1$.\\
If $ g=\sum_{i=1}^{n}\beta_iX^i $ and $ f=\sum_{i=1}^{n}f_iX^i $, then $ f\equiv g\mod p $ if and only if $ f_i\equiv \beta_i\mod p $ for $ i=0,\dots,n-1 $. This means $ f_i=\beta_i+pk_i $ for some $ k_i\in\Z $ and $ -N\le f_i\le N $. Therefore$$\frac{-N-\beta_i}{p}\le k_i\le\frac{N-\beta_i}{p};$$so for each coefficient $ f_i $ of $ f $ we have$$\Big[\frac{N-\beta_i}{p}\Big]-\Big[\frac{-N-\beta_i}{p}\Big]=\frac{2N}{p}+O(1)$$choices. Hence$$|\{f\in\mathscr{P}_{n,N}:f\equiv g\mod p\}|=\frac{(2N)^n}{p^n}+O(N^{n-1}),$$so\begin{align*}
	|\{f\in\mathscr{P}_{n,N}^0:f\equiv g\mod p\}|&=\underset{f\equiv g\tiny\mbox{ mod }p}{\sum_{f\in\mathscr{P}_{n,N}}}1+O\Big(\sum_{f\notin\mathscr{P}_{n,N}^{0}}1\Big)\\&=\frac{(2N)^n}{p^n}+O(N^{n-1}).
\end{align*}
As long as $ p^n<N $, we get\begin{align*}
	\frac{1}{|\mathscr{P}_{n,N}^{0}|}\underset{f\equiv g\tiny\mbox{ mod }p}{\sum_{f\in\mathscr{P}_{n,N}^{0}}}1&=\frac{1}{(2N)^n}(1+O(N^{n-1}))\Big(\frac{(2N)^n}{p^n}+O(N^{n-1})\Big)\\
	&=(1+O(N^{-1}))\Big(\frac{1}{p^n}+O(N^{-1})\Big)\\
	&=\frac{1}{p^n}+O(N^{-1}).
\end{align*}

Now, if $k>1$, by the Chinese Remainder Theorem we have
\begin{multline*}
	\frac{1}{|\mathscr{P}_{n,N}^{0}|}\underset{f\equiv g_i\tiny\mbox{ mod }p_i\ \forall i}{\sum_{f\in\mathscr{P}_{n,N}^{0}}}1=\left( \frac{1}{p_1^n}+O(N^{-1})\right) \dots \left( \frac{1}{p_k^n}+O(N^{-1})\right) \\
	=\frac{1}{(p_1\dots p_k)^n}+O_n\left(\sum_{1\le j_1<\dots j_{k-1}\le k-1} \frac{1}{(p_{j_1}\dots p_{j_{k-1}})^n}N^{-1}\right) \\
	=\frac{1}{(p_1\dots p_k)^n}+O_{n,k}(N^{-1}).
\end{multline*}
\end{proof}

\subsection{A Chebotarev Density Theorem on average}

Let $r=(r_1,\dots,r_n)$ be a splitting type. The constant $\delta(r)$ is defined as in \ref{eq1} as$$	\delta(r)=\prod_{i=1}^{n}\frac{1}{i^{r_i}r_i!} .$$

\begin{prop}\label{p1} One has, for all primes $ p<N^{1/(n+1)} $,
	\begin{enumerate}
		\item[$\paruno \ai\pardue$] $ \p_N(\mathbbm{1}_{p}=1)=\mathbb{E}_N(\mathbbm{1}_{p})=\delta(r)+\frac{C_{r}}{p}+O_r\Big(\frac{1}{p^2}+p^nN^{-1}\Big) $,\\where $ C_{r}=\delta(r)\frac{r_2(r_2-1)(r_1+1)(r_1+2)}{2^{r_2+1}r_1!r_2!}$;
		\item[$\paruno \bi\pardue$] $ \sigma^2_N(\mathbbm{1}_{p})=(\delta(r)-\delta(r)^2)+\frac{C_{r}(1-2\delta(r))}{p}+O_r\Big(\frac{1}{p^2}+p^nN^{-1}\Big) $. 
	\end{enumerate}It follows that, for $ x<N^{1/(n+1)} $
	\begin{enumerate}
		\item[$\paruno \ci\pardue$] $ \mathbb{E}_N(\pi_{f,r}(x))=\delta(r)\pi(x)+C_{r}\log\log x+O_n(1)$
	\end{enumerate}
\end{prop}
Hence, the \textit{normal order} of $ \pi_{f,r}(x) $ is $ \delta(r)\pi(x) $, which means that $ \pi_{f,r}(x)\sim \delta(r)\pi(x)$ for almost all $ f $'s, as $ x\rightarrow+\infty $ and $ N $ large enough. Part (c) of the above proposition is an average version of the Chebotarev Density Theorem.
\begin{proof}
	
	Fix a prime $ p $; then
	\begin{align*}
		\mathbb{E}_N(\mathbbm{1}_{p})&=\frac{1}{|\mathscr{P}_{n,N}^{0}|}\sum_{f\in\mathscr{P}_{n,N}^{0}}\mathbbm{1}_p(f)\\
		&=\frac{1}{|\mathscr{P}_{n,N}^{0}|}\underset{f\tiny{\mbox{ of splitting type }}r\tiny{\mbox{ mod }} p}{\sum_{f\in\mathscr{P}_{n,N}^{0}}}1\\
		&=\frac{1}{|\mathscr{P}_{n,N}^{0}|}\sum_{g\in 	\mathscr{X}_{n,r,p}}\underset{f\equiv g\tiny\mbox{ mod }p}{\sum_{f\in\mathscr{P}_{n,N}^{0}}}1.
	\end{align*}

	On the other hand,$$|	\mathscr{X}_{n,r,p}|=\prod_{k=1}^{n}\binom{A_{p,k}}{r_k},$$where $ A_{p,k} $ is the number of degree-$ k $ monic irreducible polynomials in $ \mathbb{F}_p[X] $, which, by the M\"{o}bius inversion formula, equals$$\frac{1}{k}\sum_{d|k}\mu(d)p^{k/d}=\frac{p^k}{k}+O(p^{\alpha_k}),$$where $\alpha_k=1$ if $k=2$, and $\alpha_k< k-1$ if $k>2$. One has, for all $ k\ge2$\begin{align*}
		\binom{A_{p,k}}{r_k}&=\frac{A_{p,k}(A_{p,k}-1)\dots(A_{p,k}-r_k+1)}{r_k!}\\
		&=\frac{1}{r_k!}\Big(\frac{p^k}{k}+O(p^{\alpha_k})\Big)\dots\Big(\frac{p^k}{k}-r_k+1+O_k(p^{\alpha_k})\Big).
	\end{align*}It turns out that$$
	\binom{A_{p,k}}{r_k}=\begin{cases}
		\frac{1}{r_1!}p(p-1)\dots(p-r_1+1)&\mbox{ if }k=1\\
		\frac{1}{r_2!2^{r_2}}p^{2r_2}+C(r_2)p^{2r_2-1}+O(p^{2r_2-2})&\mbox{ if }k=2\\
		\frac{1}{r_k!k^{r_k}}p^{kr_k}+O(p^{k(r_k-1)+\alpha_k})&\mbox{ if }k>1,
	\end{cases}
	$$
	where $C(r_2)=-\frac{r_2(r_2-1)}{2^{r_2}r_2!}$.
	Hence\begin{multline*}
		|\mathscr{X}_{n,r,p}|=\frac{1}{r_1!}p(p-1)\dots(p-r_1+1)\frac{1}{r_2!2^{r_2}}(p^{2r_2}+C(r_2)p^{2r_2-1}+O(p^{2r_2-2}))\\\prod_{k=3}^{n}(\frac{1}{r_k!k^{r_k}}p^{kr_k}+O(p^{k(r_k-1)+\alpha_k})),
	\end{multline*}so
	\begin{equation}
			|\mathscr{X}_{n,r,p}|=\delta(r)p^{n}+C_rp^{n-1}+O_{n,r}(p^{n-2}),
			\label{eq2}
	\end{equation}
	where $ C_r=-\delta(r)C(r_2)\frac{(r_1+1)(r_1+2)}{2r_1!}.$
	
	By Lemma \ref{l1}, for $ p^{n+1}<N$,\begin{align*}
		\mathbb{E}_N(\mathbbm{1}_{p})&=(\delta(r)p^n+C_{r}p^{n-1}+O(p^{n-2}))\Big(\frac{1}{p^n}+O(N^{-1})\Big)\\
		&=\delta(r)+\frac{C_{r}}{p}+O\Big(\frac{1}{p^2}+p^nN^{-1}\Big),
	\end{align*}which proves (a). Part (b) follows by (a) and the fact that the random variable is either 0 or 1.
	
	For (c), by linearity, we simply have to sum over all primes $ p\le x $ and use the elementary estimate$$\sum_{p\le x}\frac{1}{p}=\log\log x+O(1)$$to get$$\mathbb{E}_N(\pi_{f,r}(x))=\delta(r)\pi(x)+C_{r}\log\log x+O\left( 1+\tfrac{x^{n+1}}{\log x}N^{-1}\right) $$by partial summation. By the assumption on the size of $x<N^{1/(n+1)} $, the term $ \tfrac{x^{n+1}}{\log x}N^{-1} $ is negligible, and we have the claim.
\end{proof}

\begin{osserv}
	 From (c) of Proposition \ref{p1}, we have that for every $m\ge2$, $x<N^{1/(n+1)} $,$$\pi_{f,r}(x)-\delta(r)\pi(x)=O((\log\log x)^m),$$for all but $O_{n}\Big(x^{(n+1)(n-1)}(\log\log x)^{1-m}\Big)$ $S_n$-polynomials $f$ of height $\ll x^{n+1}$. This saves over the total number of polynomials, which is $\asymp_n N^n$.
\end{osserv}

 For $f\in\mathscr{P}_{n,N}^{0}$, Let $\varphi:G_f\rightarrow\Co$ be a class function, i.e. constant on conjugacy classes. 
Define$$\pi_{f,\varphi}(x)=\sum_{p\le x,\ p\nmid D_f}\varphi(\Frob_{f,p}).$$Then, if we sum over the conjugacy classes, that is over the splitting types $r=(r_1,\dots,r_n)$ we get$$\pi_{f,\varphi}(x)=\sum_{r}\varphi(g_r)\pi_{f,r}(x),$$where $g_r$ is any element of the conjugacy class $\mathscr{C}_r$ for every $r$. 
\begin{coroll}\label{c2}
	If $x<N^{1/(n+1)}$,$$\mathbb{E}_N(\pi_{f,\varphi}(x))=\sum_{r}\delta(r)\varphi(g_r)\pi(x)+\sum_{r}\delta(r)\varphi(g_r)\log\log x+O(||\varphi||)$$where $||\varphi||=\sup_{g\in G_f}|\varphi(g)|$, and the implied constant depends on the number of conjugacy classes.
\end{coroll}

\subsection{Higher moments}

The following approach is motivated by the proof of the Erd\H{o}s-Kac Theorem in \cite{GS}, where they compute the moments$$\sum_{n\le x}(\omega(n)-\log\log x)^k$$of the prime divisor function, uniformely in a wide range of $ k $. 
Fix a splitting type $ r $ and a prime $ p $. Consider the independent discrete random variables $ X_p $ defined by$$\p(X_p=1)=\frac{|	\mathscr{X}_{n,r,p}|}{p^n}.$$In \ref{eq2} we showed that

\begin{equation}
	\p(X_p=1)=\frac{|	\mathscr{X}_{n,r,p}|}{p^n}=\delta(r)+\frac{C_{r}}{p}+O\left( \frac{1}{p^2}\right) .
\end{equation}For all prime $p$ we define the function$$Y_p(f)=\begin{cases}
	1-\frac{|	\mathscr{X}_{n,r,p}|}{p^n}&\mbox{ if }\mathbbm{1}_p(f)=1,\\
	-\frac{|	\mathscr{X}_{n,r,p}|}{p^n}&\mbox{ otherwise }.
\end{cases}$$Now, we consider a generalization $Y_a$ of the function $Y_p$ for any natural number $a>0$, whose $ k $-moments are small unless $ a $ is a square-full number, that is, it satisfies the following property: $ p^\alpha||a\Rightarrow\alpha\ge 2 $.

In the next lemma we compute an asymptotic for the moments of $Y_p(f)$. We explain here why this yields to the moments of $\pi_{f,r}(x)-\delta(r)\pi(x)$, the claim of Proposition \ref{p2}.
Observe that, for any real number $ z<x $, we can write\begin{align*}
	\pi_{f,r}(x)-\delta(r)\pi(x)&=\sum_{p\le z}Y_p(f)+\sum_{z<p\le x}\mathbbm{1}_p(f)\\&+\left( \sum_{p\le z}\frac{|	\mathscr{X}_{n,r,p}|}{p^n}-\delta(r)\pi(x)\right) .
\end{align*}

The Prime Number Theorem has an error term $O(xe^{-c\sqrt{\log x}})$, which is prohibitive for counting primes in short intervals. We then use the trivial estimate $\sum_{z<p\le x}\mathbbm{1}_p(f)\le x-z$ when $z$ is close to $x$. Pick now $z=x-k$ ($k\le x$). Then$$\sum_{z<p\le x}\mathbbm{1}_p(f)\le k,$$and, by \ref{eq2},\begin{align*}
	\sum_{p\le z}\frac{|	\mathscr{X}_{n,r,p}|}{p^n}-\delta(r)\pi(x)\ll_r|\pi(x)-\pi(z)|+\log\log z\ll k+\log\log z.
\end{align*}

By the above one has\begin{equation}
	\pi_{f,r}(x)-\delta(r)\pi(x)=\sum_{p\le z}Y_p(f)+O_r(k+\log\log z).
	\label{eq3}
\end{equation}

\begin{lemma}\label{l2}
	Uniformly for \textbf{even} natural numbers $ k $ with $$ k\ll_{n,r}\min\left(\left( \frac{z}{\log z\log\log z}\right)^{1/3}, \frac{\log N}{\log z}\right) , $$one has\begin{multline*}
		\mathbb{E}_N\Big( \Big( \sum_{p\le z}Y_p(f)\Big) ^k\Big) =C_{k,r}\pi(z)^{k/2}\left( 1+O_{n,r}\left( \frac{k^3}{(1-\delta(r))^{k/2}}\frac{\log\log z}{\pi(z)}\right) \right) \\+O_n\left( \frac{z^{k(n+1)}}{\log z}N^{-1}\right) .
	\end{multline*}While uniformly for \textbf{odd} natural numbers $ k $ with $$k\ll_{n,r}\min\left( \frac{z}{\log z\log\log z}, \frac{\log N}{\log z}\right),$$one has $$\mathbb{E}_N\Big( \Big( \sum_{p\le z}Y_p(f)\Big) ^k\Big)\ll_{n,r} C_{k,r}\pi(z)^{k/2}k\frac{\log\log z}{\pi(z)^{1/2}}+\frac{z^{k(n+1)}}{\log z}N^{-1},$$as $ z,N\rightarrow+\infty $.
	
	Here$$C_{k,r}=\begin{cases}
		(\delta(r)-\delta(r)^2)^{k/2}	\frac{k!}{2^{k/2}(k/2)!}&\mbox{ for }k\mbox{ even}\\
		\delta(r)^{\frac{k-1}{2}}\frac{k!}{2^{\frac{k-1}{2}}(\frac{k-1}{2})!}&\mbox{ for }k\mbox{ odd}.
	\end{cases}$$
\end{lemma}

\begin{proof}
	Let $ a=\prod_{i=1}^{s}p_i^{\alpha_i} $, where the $ p_i $ are distinct primes and $ \alpha_i\ge1 $. Let $ A:=\prod_{i=1}^{s}p_i $ be the square-free part of $ a $, and set$$Y_a(f):=\prod_{i=1}^{s}Y_{p_i}(f)^{\alpha_i}.$$
	We may write$$\frac{1}{|\mathscr{P}_{n,N}^{0}|}\sum_{f\in\mathscr{P}_{n,N}^{0}}(\sum_{p\le z}Y_p(f))^k=\sum_{p_1,\dots,p_k\le z}\frac{1}{|\mathscr{P}_{n,N}^{0}|}\sum_{f\in\mathscr{P}_{n,N}^{0}}Y_{p_1\dots p_k}(f).$$Let us then consider more generally $ \frac{1}{|\mathscr{P}_{n,N}^{0}|}\sum_{f\in\mathscr{P}_{n,N}^{0}}Y_a(f) $.
	By definition, for any prime $ p $, $ Y_p(f)=Y_p(g) $ if $ f\equiv g $ mod $ p $; therefore\begin{align*}
		\frac{1}{|\mathscr{P}_{n,N}^{0}|}\sum_{f\in\mathscr{P}_{n,N}^{0}}Y_a(f)&=\frac{1}{|\mathscr{P}_{n,N}^{0}|}\underset{i=1,\dots,s}{\sum_{g_i\tiny{\mbox{ mod }}p_i}}\underset{f\equiv g_i\tiny{\mbox{ mod }}p_i\ \forall i}{\sum_{f\in\mathscr{P}_{n,N}^{0}}}Y_{p_1}(g_1)^{\alpha_1}\dots Y_{p_s}(g_s)^{\alpha_s},
	\end{align*}where the first sum in the right-hand side is over $ g_i\in\Z[X] $ monic, with coefficients in $ \{0,\dots,p_i-1\} $. As long as $ (p_1\dots p_s)^n<N $, the sum is, by Lemma \ref{l1},\begin{multline*}
		\sum_{g_1,\dots,g_s}Y_{p_1}(g_1)^{\alpha_1}\dots Y_{p_s}(g_s)^{\alpha_s}	\frac{1}{|\mathscr{P}_{n,N}^{0}|}\underset{f\equiv g_i\tiny{\mbox{ mod }}p_i\ \forall i}{\sum_{f\in\mathscr{P}_{n,N}^{0}}}1\\
		=\sum_{g_1,\dots,g_s}Y_{p_1}(g_1)^{\alpha_1}\dots Y_{p_s}(g_s)^{\alpha_s}\Big(\frac{1}{(p_1\dots p_s)^n}+O(N^{-1})\Big)\\
		=\frac{1}{A^n}\sum_{g_1,\dots,g_s}Y_{p_1}(g_1)^{\alpha_1}\dots Y_{p_s}(g_s)^{\alpha_s}	+ O\Big(N^{-1}\sum_{g_1,\dots,g_s}1\Big)\\
		=\frac{1}{A^n}\sum_{g_1,\dots,g_s}Y_{p_1}(g_1)^{\alpha_1}\dots Y_{p_s}(g_s)^{\alpha_s}	+ O(A^nN^{-1}),
	\end{multline*}since $ |Y_{p_i}(g_i)^{\alpha_i}|\ll 1 $. Denoting the main term by 
	
	\begin{equation}
		Y(a)=\frac{1}{A^n}\sum_{g_1,\dots,g_s}Y_{p_1}(g_1)^{\alpha_1}\dots Y_{p_s}(g_s)^{\alpha_s},
		\label{eq4}
	\end{equation}

	we have\begin{multline}
		Y(a)=\frac{1}{A^n}\sum_{g_1}Y_{p_1}(g_1)^{\alpha_1}\dots \sum_{g_s}Y_{p_s}(g_s)^{\alpha_s}\\
		=\frac{1}{A^n}\prod_{i=1}^{s}\Big(\sum_{g_i\in X_{n,r,p_i}}\Big(1-\frac{|	\mathscr{X}_{n,r,p_i}|}{p_i^n}\Big)^{\alpha_i}+\sum_{g_i\notin X_{n,r,p_i}}\Big(-\frac{|	\mathscr{X}_{n,r,p_i}|}{p_i^n}\Big)^{\alpha_i}\Big)\\
		=\frac{1}{A^n}\prod_{i=1}^{s}\Big(|	\mathscr{X}_{n,r,p_i}|\Big(1-\frac{|	\mathscr{X}_{n,r,p_i}|}{p_i^n}\Big)^{\alpha_i}+(p_i^{n}-|X_{n,r,p_i}|)\Big(-\frac{|	\mathscr{X}_{n,r,p_i}|}{p_i^n}\Big)^{\alpha_i}\Big)\\
		=\prod_{p^{\alpha}||a}\Big(\frac{|	\mathscr{X}_{n,r,p}|}{p^n}\Big(1-\frac{|	\mathscr{X}_{n,r,p}|}{p^n}\Big)^{\alpha}+\Big(1-\frac{|	\mathscr{X}_{n,r,p}|}{p^n}\Big)\Big(-\frac{|	\mathscr{X}_{n,r,p}|}{p^n}\Big)^{\alpha}\Big)
		\label{eq5},
	\end{multline}using the inductive formula $$\prod_{i=1}^{\ell}(a_i+b_i)=\underset{k+h=\ell}{\underset{j_1<\dots<j_h\in\{1,\dots,\ell\}\setminus\{i_1,\dots,i_k\}}{\sum_{1\le i_1<\dots<i_k\le \ell}}}a_{i_1}\dots a_{i_k}b_{j_1}\dots b_{j_h}.$$

	Thus\begin{equation*}
		\frac{1}{|\mathscr{P}_{n,N}^{0}|}\sum_{f\in\mathscr{P}_{n,N}^{0}}Y_a(f)=Y(a)+O(A^nN^{-1});
	\end{equation*}Observe now that $ Y(a)=0 $ unless $ \alpha_i\ge 2 $ for all $ i=1,\dots,s $. It turns out that\begin{multline}
		\frac{1}{|\mathscr{P}_{n,N}^{0}|}\sum_{f\in\mathscr{P}_{n,N}^{0}}(\sum_{p\le z}Y_p(f))^k=\underset{p_1\dots p_k\ \tiny\mbox{square-full}}{\sum_{p_1,\dots,p_k\le z}}Y(p_1\dots p_k)\\+O\Big(\sum_{p_1,\dots,p_k\le z}(p_1\dots p_k)^nN^{-1}\Big)\\
		=\underset{p_1\dots p_k\ \tiny\mbox{square-full}}{\sum_{p_1,\dots,p_k\le z}}Y(p_1\dots p_k)+O_n\left( \frac{z^{k(n+1)}}{\log z}N^{-1}\right) .
		\label{eq6}
	\end{multline}Let $ q_1<\dots<q_s $ be the distinct primes in $ p_1\dots p_k $. Since $ p_1\dots p_k $ is square-full, we have $ s\le k/2 $. The main term above is\begin{equation}
		\sum_{s\le k/2}\sum_{q_1<\dots<q_s\le z}\underset{\alpha_1+\dots+\alpha_s=k}{\sum_{\alpha_1,\dots,\alpha_s\ge2}}\binom{k}{\alpha_1,\dots,\alpha_s}Y(q_1^{\alpha_1}\dots q_s^{\alpha_s}).
		\label{eq7}
	\end{equation}At this point, we have to divide into two cases, since if $ k $ is even there is a term $ s=k/2 $ with all $ \alpha_i=2 $. This main term contributes\begin{multline*}
		\frac{k!}{2^{k/2}(k/2)!}\underset{q_j\tiny{\mbox{ distinct}}}{\sum_{q_1,\dots,q_{k/2}\le z}}Y(q_1^2\dots q_{k/2}^2)\\
		=\frac{k!}{2^{k/2}(k/2)!}\underset{q_j\tiny{\mbox{ distinct}}}{\sum_{q_1,\dots,q_{k/2}\le z}}\prod_{i=1}^{k/2}\frac{|	\mathscr{X}_{n,r,q_i}|}{q_i^n}\Big(1-\frac{|\mathscr{X}_{n,r,q_i}|}{q_i^n}\Big).
	\end{multline*}Now,$$\underset{q_j\tiny{\mbox{ distinct}}}{\sum_{q_1,\dots,q_{k/2}\le z}}\prod_{i=1}^{k/2}\frac{|	\mathscr{X}_{n,r,q_i}|}{q_i^n}\Big(1-\frac{|	\mathscr{X}_{n,r,q_i}|}{q_i^n}\Big)\le\Big(\sum_{p\le z}\frac{|	\mathscr{X}_{n,r,p}|}{p^n}\Big(1-\frac{|	\mathscr{X}_{n,r,p}|}{p^n}\Big)\Big)^{k/2}.$$On the other hand, by induction,\begin{multline*}
		\underset{q_j\tiny{\mbox{ distinct}}}{\sum_{q_1,\dots,q_{k/2}\le z}}\prod_{i=1}^{k/2}\frac{|	\mathscr{X}_{n,r,q_i}|}{q_i^n}\Big(1-\frac{|	\mathscr{X}_{n,r,q_i}|}{q_i^n}\Big)\\
		=\underset{q_j\tiny{\mbox{ distinct}}}{\sum_{q_1,\dots,q_{k/2-1}\le z}}\prod_{i=1}^{k/2-1}\frac{|	\mathscr{X}_{n,r,q_i}|}{q_i^n}\Big(1-\frac{|	\mathscr{X}_{n,r,q_i}|}{q_i^n}\Big)\underset{q_{k/2}\neq q_j\ \forall j}{\sum_{q_{k/2}\le z}}\frac{|	\mathscr{X}_{n,r,q_{k/2}}|}{q_{k/2}^n}\Big(1-\frac{|	\mathscr{X}_{n,r,q_{k/2}}|}{q_{k/2}^n}\Big)\\
		\ge\underset{q_j\tiny{\mbox{ distinct}}}{\sum_{q_1,\dots,q_{k/2-1}\le z}}\prod_{i=1}^{k/2-1}\frac{|	\mathscr{X}_{n,r,q_i}|}{q_i^n}\Big(1-\frac{|	\mathscr{X}_{n,r,q_i}|}{q_i^n}\Big)\sum_{\pi_{k/2}\le p\le z}\frac{|	\mathscr{X}_{n,r,p}|}{p^n}\Big(1-\frac{|	\mathscr{X}_{n,r,p}|}{p^n}\Big)\\
		\ge\dots\ge \sum_{ 2\le p\le z}\frac{|	\mathscr{X}_{n,r,p}|}{p^n}\Big(1-\frac{|	\mathscr{X}_{n,r,p}|}{p^n}\Big)\dots\sum_{\pi_{k/2}\le p\le z}\frac{|	\mathscr{X}_{n,r,p}|}{p^n}\Big(1-\frac{|	\mathscr{X}_{n,r,p}|}{p^n}\Big)\\
		\ge \Big(\sum_{\pi_{k/2}\le p\le z}\frac{|	\mathscr{X}_{n,r,p}|}{p^n}\Big(1-\frac{|	\mathscr{X}_{n,r,p}|}{p^n}\Big)\Big)^{k/2},
	\end{multline*}where $\pi_n$ denotes the $n$-th smallest prime. The first inequality is due to the observation that if a prime $p$ is greater than the $n$-th smallest prime, then there are $n-1$ primes $p_1,\dots,p_{n-1}$ so that $p\neq p_i$ for all $i=1,\dots,n-1$.

	By (1)\begin{align*}
		\sum_{p\le z}\frac{|	\mathscr{X}_{n,r,p}|}{p^n}\Big(1-\frac{|	\mathscr{X}_{n,r,p}|}{p^n}\Big)&=(\delta(r)-\delta(r)^2)\pi(z)+O(\log\log z),\\
		\sum_{\pi_{k/2}\le p\le z}\frac{|	\mathscr{X}_{n,r,p}|}{p^n}\Big(1-\frac{|	\mathscr{X}_{n,r,p}|}{p^n}\Big)&=(\delta(r)-\delta(r)^2)\pi(z)+O(\log\log z+k).
	\end{align*}The main term in (2) is then\begin{multline}
		\frac{k!}{2^{k/2}(k/2)!}((\delta(r)-\delta(r)^2)\pi(z)+O(\log\log z+k))^{k/2}\\
		=\frac{k!}{2^{k/2}(k/2)!}(\delta(r)-\delta(r)^2)^{k/2}(\pi(z)^{k/2}+O(k^2\pi(z)^{k/2-1}\log\log z)).
		\label{eq8}
	\end{multline}We have now to estimate the error term in (2), for $ s=k/2-1$. Note that  $Y(q_1^{\alpha_1}\dots q_s^{\alpha_s})\le Y(q_1^{2}\dots q_s^{2})$ for $\alpha_i\ge2$. By \ref{eq5} we have $$ Y(q_1^{\alpha_1}\dots q_s^{\alpha_s})\le\frac{|	\mathscr{X}_{n,r,q_1}|\dots|	\mathscr{X}_{n,r,q_s}|}{(q_1\dots q_s)^n} .$$Therefore
	\begin{multline}
		\sum_{q_1<\dots<q_s\le z}\underset{\alpha_1+\dots+\alpha_s=k}{\sum_{\alpha_1,\dots,\alpha_s\ge2}}\binom{k}{\alpha_1,\dots,\alpha_s}Y(q_1^{\alpha_1}\dots q_s^{\alpha_s})\\
		\le\frac{k!}{(k/2-1)!}\Big(\sum_{q\le z}\frac{|	\mathscr{X}_{n,r,q}|}{q^n}\Big)^{k/2-1}\underset{\alpha_1+\dots+\alpha_{k/2-1}=k}{\sum_{\alpha_1,\dots,\alpha_{k/2-1}\ge 2}}\frac{1}{\alpha_1!\dots\alpha_{k/2-1}!}\\
		\le\frac{k!}{2^{k/2-1}(k/2-1)!}\binom{k/2}{k/2-2}(\delta(r)\pi(z)+O(\log\log z))^{k/2-1} \\
		\ll \frac{k!}{2^{k/2}(k/2)!} k^3\left( \delta(r)^{k/2-1}\pi(z)^{k/2-1}+k\pi(z)^{k/2-2}\log\log z\right)\\
		\ll \frac{k!}{2^{k/2}(k/2)!} k^3\delta(r)^{k/2-1}\pi(z)^{k/2-1}. 
		\label{eq9}
	\end{multline}For the last inequality, we use the fact that the number of tuples of integers $ (\alpha_1,\dots,\alpha_{k/2-1}) $, $ \alpha_i\ge 2 $ such that $ \sum\alpha_i=k $ is the number of sequences $ (\alpha_1',\dots,\alpha_{k/2-1}') $, $ \alpha_i'\ge 1 $ such that $ \sum\alpha_i=k/2+1 $, that is the number of strong compositions of $ k/2+1 $ into $ k/2-1 $ parts, which is $ \binom{k/2}{k/2-2} $. Thus, for $ k $ even, by combining \ref{eq6}, \ref{eq8} and \ref{eq9},\begin{multline*}
		\frac{1}{|\mathscr{P}_{n,N}^{0}|}\sum_{f\in\mathscr{P}_{n,N}^{0}}(\sum_{p\le z}Y_p(f))^k\\
		= \frac{k!}{2^{k/2}(k/2)!}(\delta(r)-\delta(r)^2)^{k/2}(\pi(z)^{k/2}+O_{n,r}(k^2\pi(z)^{k/2-1}\log\log z\\+\frac{k^3}{(1-\delta(r))^{k/2}}\pi(z)^{k/2-1}))
		+O_n\left( \frac{z^{k(n+1)}}{\log z}N^{-1}\right) \\
		=\frac{k!}{2^{k/2}(k/2)!}(\delta(r)-\delta(r)^2)^{k/2}\pi(z)^{k/2}\left( 1+O_{n,r}\left( \frac{k^3}{(1-\delta(r))^{k/2}}\frac{\log\log z}{\pi(z)}\right) \right) \\+O_n\left( \frac{z^{k(n+1)}}{\log z}N^{-1}\right).
	\end{multline*}Finally, for $ k $ odd, we have the estimate for the term with $ s=k/2-1/2 $ as for the previous case, obtaining\begin{multline*}
	\frac{1}{|\mathscr{P}_{n,N}^{0}|}\sum_{f\in\mathscr{P}_{n,N}^{0}}\left( \sum_{p\le z}Y_p(f)\right)  ^k\\
	\ll \frac{k!}{2^{\frac{k-1}{2}}(\frac{k-1}{2})!}k\left(\delta(r)^{\frac{k-1}{2}}\pi(z)^{\frac{k-1}{2}}+k\pi(z)^{\frac{k-3}{2}}\log\log z\right) \\+\frac{z^{k(n+1)}}{\log z}N^{-1}\\
	\ll_{n,r} C_{k,r}\pi(z)^{k/2}k\frac{\log\log z}{\pi(z)^{1/2}}+\frac{z^{k(n+1)}}{\log z}N^{-1}.
\end{multline*}
\end{proof}
\begin{prop}\label{p2} Let $x=N^{1/\log\log N}$.
	Uniformly for \textbf{even} natural numbers $ k $ with $ k\ll_{n,r}\log\log N $, one has\begin{multline*}
		\mathbb{E}_N((\pi_{f,r}(x)-\delta(r)\pi(x))^k)\\=C_{k,r}\pi(x)^{k/2}\left( 1+O_{n,r}\left( \frac{1}{(1-\delta(r))^{k/2}}\frac{\log\log x}{\pi(x)^{1/2}}\right) \right).
	\end{multline*}While uniformly for \textbf{odd} natural numbers $ k $ with $k\ll_{n,r}\log\log N$, one has$$\mathbb{E}_N((\pi_{f,r}(x)-\delta(r)\pi(x))^k)\ll_{n,r} C_{k,r}\pi(x)^{k/2}k\frac{\log\log x}{\pi(x)^{1/2}},$$as $ x,N\rightarrow+\infty $.
\end{prop}
\begin{proof}
	For $ z=x-k $ we obtained in \ref{eq3},$$\pi_{f,r}(x)-\delta(r)\pi(x)=\sum_{p\le z}Y_p(f)+O_r(k+\log\log z) .$$In particular, $$
	(\pi_{f,r}(x)-\delta(r)\pi(x))^k=\left( \sum_{p\le z}Y_p(f)\right) ^k$$\begin{equation}+O\left( \sum_{j=0}^{k-1}(k+\log\log z) ^{k-j}\binom{k}{j}\left| \sum_{p\le z}Y_p(f)\right| ^j\right) 
		\label{eq10}
	\end{equation}The dominant term in the error is obtained for $ j=k-1 $. If $k$ is even, we apply Lemma \ref{l2} to (3) and we get\begin{multline*}
		\frac{1}{|\mathscr{P}_{n,N}^{0}|}\sum_{f\in\mathscr{P}_{n,N}^{0}}(\pi_{f,r}(x)-\delta(r)\pi(x))^k\\
		=C_{k,r}\pi(x-k)^{k/2}\left( 1+O_{n,r}\left( \frac{k^3}{(1-\delta(r))^{k/2}}\frac{\log\log(x-k)}{\pi(x-k)}+k^3\frac{C_{k-1,r}}{C_{k,r}}\frac{(\log\log(x-k))^2}{\pi(x-k)}\right) \right)\\+O_n\left( \frac{(x-k)^{k(n+1)}}{\log(x-k)}N^{-1}\right) \\
		=C_{k,r}\pi(x)^{k/2}\left( 1+O_{n,r}\left( \frac{k^3}{(1-\delta(r))^{k/2}}\frac{\log\log x}{\pi(x)}+\frac{k^3(\log\log x)^2}{\pi(x)} +\frac{k^2}{\pi(x)}\right)\right) 
		 \\+O_n\left( \frac{x^{k(n+1)}}{\log x}N^{-1}\right) ,
	\end{multline*}
	since\begin{align*}
		&\pi(x-k)=\pi(x)-\sum_{x-k<p<x}1=\pi(x)+O(k)\\
		\Longrightarrow&\pi(x-k)^{k/2}=\pi(x)^{k/2}+O(\pi(x)^{k/2-1}k^2),
	\end{align*}and$$\frac{C_{k-1,r}}{C_{k,r}}\ll_r1.$$ The last error term is negligible if $k\ll\log\log N$.

	If $k$ is odd, we can handle it using the Cauchy-Schwartz inequality:\begin{multline*}
		\frac{1}{|\mathscr{P}_{n,N}^{0}|}\sum_{f\in\mathscr{P}_{n,N}^{0}}\Big|\sum_{p\le z}Y_p (f)\Big|^{k-1}\\
		\le\Big(\frac{1}{|\mathscr{P}_{n,N}^{0}|}\sum_{f\in\mathscr{P}_{n,N}^{0}}\Big|\sum_{p\le z}Y_p(f)\Big|^{k-2}\Big)^{1/2}\Big(\frac{1}{|\mathscr{P}_{n,N}^{0}|}\sum_{f\in\mathscr{P}_{n,N}^{0}}\Big|\sum_{p\le z}Y_p(f)\Big|^{k}\Big)^{1/2}.
	\end{multline*}Lemma \ref{l2} leads to$$\frac{1}{|\mathscr{P}_{n,N}^{0}|}\sum_{f\in\mathscr{P}_{n,N}^{0}}\Big|\sum_{p\le z}Y_p(f)\Big|^{k-1}\ll_{n,r}(C_{k-2,r}C_{k,r})^{1/2}k\pi(z)^{\frac{k}{2}-1}\log\log z+\frac{x^{k(n+1)}}{\log x}N^{-1}.$$Since$$\frac{(C_{k-2,r}C_{k,r})^{1/2}}{C_{k,r}}\binom{k}{k-1}\asymp k^{1/2},$$we obtain from (3), the condition $k\ll\log\log N$, and the assumption on $x$,
	\begin{multline*}
		\frac{1}{|\mathscr{P}_{n,N}^{0}|}\sum_{f\in\mathscr{P}_{n,N}^{0}}(\pi_{f,r}(x)-\delta(r)\pi(x))^k\ll_{n,r} C_{k,r}\pi(x)^{k/2}\log\log x\\
		\left( \frac{k}{\pi(x)^{1/2}}+\frac{k^{5/2}\log\log x}{\pi(x)}\right) +\frac{x^{k(n+1)}}{\log x}N^{-1}\\
		\ll_{n,r} C_{k,r}\pi(x)^{k/2}k\frac{\log\log x}{\pi(x)^{1/2}}.
	\end{multline*}

\end{proof}

\section{Proof of the main theorem}

\begin{thm}\label{t1}
	For $ x\le N^{1/\log\log N} $ and for any $ a<b\in\R $,$$\p_N\left(a\le \frac{\pi_{f,r}(x)-\delta(r)\pi(x)}{(\delta(r)-\delta(r)^2)^{1/2}\pi(x)^{1/2}}\le b\right) \longrightarrow\frac{1}{\sqrt{2\pi}}\int_{a}^{b}e^{-t^2/2}dt,$$as $ N\rightarrow+\infty $.
\end{thm}
\begin{proof}
Firstly, note that it is equivalent to say that$$\p_N\left( \frac{\pi_{f,r}(x)-\delta(r)\pi(x)}{(\delta(r)-\delta(r)^2)^{1/2}\pi(x)^{1/2}}\le b\right) \longrightarrow\Phi(b)$$as $ N\rightarrow+\infty $, where $\Phi(b)=\frac{1}{\sqrt{2\pi}}\int_{-\infty}^{b}e^{-t^2/2}dt$. We use the method of moments. Since the function $ \Phi $ is determined by its moments$$\mu_k=\int_{-\infty}^{+\infty}x^kd\Phi(x),$$if a family of distribution functions $ F_n $ satisfies $\int_{-\infty}^{+\infty}x^kdF_n(x)\rightarrow\mu_k$ for all $ k\ge 1 $, then $ F_n(x)\rightarrow\Phi(x) $ pointwise (see \cite{Fel}, p. 262). On the other hand, if $ F_n(x)\rightarrow\Phi(x) $ for each $ x $ and if $ \int_{-\infty}^{+\infty}|x|^{k+\varepsilon}dF_n(x) $ is bounded in $ n $ for some $ \varepsilon>0 $, then $\int_{-\infty}^{+\infty}x^kdF_n(x)\rightarrow\mu_k$ (\cite{Fel}, page 245).

So the theorem will follow by the method of moments if we prove that for $ k\ge1 $,$$\mathbb{E}_N\left( \frac{(\pi_{f,r}(x)-\delta(r)\pi(x))^k}{((\delta(r)-\delta(r)^2)^{1/2}\pi(x)^{1/2})^k}\right)$$converges to $ \mu_k $ as $ N\rightarrow+\infty $. 

It is well known that$$\mu_k=\frac{1}{\sqrt{2\pi}}\int_{-\infty}^{+\infty}x^ke^{-x^2/2}dx=\begin{cases}
	\frac{k!}{2^{k/2}(k/2)!}&\mbox{if }k\mbox{ is even},\\
	0&\mbox{if }k\mbox{ is odd}.
\end{cases}$$From Proposition \ref{p2}, if we fix $k\ge1$, we see exactly that$$\frac{1}{|\mathscr{P}_{n,N}^{0}|}\sum_{f\in\mathscr{P}_{n,N}^{0}}\left( \frac{(\pi_{f,r}(x)-\delta(r)\pi(x))^k}{((\delta(r)-\delta(r)^2)^{1/2}\pi(x)^{1/2})^k}\right)\underset{x\rightarrow+\infty}{\longrightarrow}	\frac{k!}{2^{k/2}(k/2)!}=\mu_k$$if $k$ is even, and$$\frac{1}{|\mathscr{P}_{n,N}^{0}|}\sum_{f\in\mathscr{P}_{n,N}^{0}}\left( \frac{(\pi_{f,r}(x)-\delta(r)\pi(x))^k}{((\delta(r)-\delta(r)^2)^{1/2}\pi(x)^{1/2})^k}\right)\ll_{k,r}\frac{\log\log x}{\pi(x)^{1/2}}\underset{x\rightarrow+\infty}{\longrightarrow}0=\mu_k$$if $k$ is odd.
\end{proof}

\section{Applications}

\subsection{Discriminant and average of ramified primes}

Let $f$ be an $S_n$-polynomial and let $d_f$ be its discriminant. We are going to discuss the relation between the number of primes $p$ dividing $d_f$ and the discriminant of its field $\Q_f/\Q$.\\

For a polynomial $ f\in \mathscr{P}_{n,N} $, the bound$$d_f\ll N^{2n-2}$$holds, since $ d_f $ is given by the $ (2n-1) $-dimensional determinant$$d_f=(-1)^{n(n-1)/2}\det\begin{pmatrix}
	1&a_{n-1}&a_{n-2}&\cdots&a_0&0&\cdots\\
	0&a_n&a_{n-1}&\cdots&a_1&a_1&\cdots\\
	\vdots&&&&&&\vdots\\
	0&\cdots&0&a_n&\cdots&a_1&a_0\\
	n&(n-1)a_{n-1}&(n-2)a_{n-2}&\cdots&0&0&\cdots\\
	\vdots&&&&&&\vdots\\
	0&\cdots&\cdots&0&na_n&\cdots&a_1
\end{pmatrix}$$with $ a_n=1 $ in our case. However, it turns out (see \cite{GZ}, Corollary 2.2) that\begin{equation}
 d_f\asymp N^{2n-2}
 \label{eq12}
\end{equation}
for almost all $ f $. Indeed, for all $ \varepsilon>0 $ there exists $ \delta=\delta(n,\varepsilon) $ s.t. for $ N $ large enough$$
\p_N(|d_f|>\delta N^{2n-2})>1-\varepsilon.
$$

By the primitive element theorem, we know there is an integral element $\theta\in\mathcal{O}_{K_f}$ so that $K_f=\Q(\theta)$. Let $f_{\theta}\in \Z[X]$ be the minimal polynomial of $\theta$. Then it holds the following relation between the discriminant of $f_{\theta}$ and the discriminant $D_f$ of the number field extension $K_f/\Q$:$$d_{f_{\theta}}=a_{f_{\theta}}^2\cdot D_f,$$where $a_{f_{\theta}}\in\Z$ (see \cite{La}, Chapter III).\\
Now, let $\alpha$ be a root of $f\in\mathscr{P}_{n,N}^{0}$ and consider the extension $\Q_f=\Q(\alpha)$ generated by $\alpha$ over $\Q$. $$\begin{tikzpicture}
	\matrix (m) [matrix of math nodes,row sep=1em,column sep=0.001em,minimum width=2em]
	{K_f\\
		\Q_f\\
		\Q\\};
	\path[-]
	(m-1-1) edge node [right] {$\scriptstyle{(n-1)!}$} (m-2-1)
	(m-2-1) edge node [right] {$\scriptstyle{n}$} (m-3-1)
	;
\end{tikzpicture}$$By the transitivity of the discriminant in towers of extensions, one has$$D_f=D_{\Q_f}^{(n-1)!}N_{\Q_f/\Q}\mathfrak{D}_{K_f/\Q_f},$$where $\mathfrak{D}_{K_f/\Q_f}$ is the ideal discriminant of the extension $K_f/\Q_f$. By the above, \begin{equation}
d_f=a_{f}^2\cdot D_{\Q_f},
\label{eq11}
\end{equation}with $a_f\in\Z$. It turns out the relation between $d_f$ and $D_f$:$$
	d_f=D_f^{1/(n-1)!}a_f^2(N_{\Q_f/\Q}\mathfrak{D}_{K_f/\Q_f})^{-1/(n-1)!}.
$$

As in Proposition 6.4 of \cite{ABZ}, we see that the probability that a monic, irreducible, degree $ n $ polynomial with height $ \le N $ has discriminant coprime with $p $ is $ 1-\frac{1}{p} $, hence$$\frac{|\{f\in\mathscr{P}_{n,N}^{\tiny\mbox{irr}}:p|d_f\}|}{|\mathscr{P}_{n,N}^{\tiny\mbox{irr}}|}\underset{N\rightarrow+\infty}{\longrightarrow}\frac{1}{p}.$$

\begin{coroll}
	The average of the number of ramified primes in $\Q_f/\Q$ is$$\mathbb{E}_N(|\{p:p|D_{\Q_f}\}|)\le\log\log N + O_n(1),$$for sufficiently large $N$.
\end{coroll}

\begin{proof}
	Since almost all polynomials in $ \mathscr{P}_{n,N} $ are irreducible, with error term $ O(N^{-1}) $, and since $|\mathscr{P}_{n,N}^{0}|=(2N)^{n}+O(N^{n-1})$, we also have that$$\frac{|\{f\in\mathscr{P}_{n,N}^{0}:p|d_f\}|}{|\mathscr{P}_{n,N}^{0}|}=\frac{1}{p}+o(1),$$for $N$ large enough. In particular, for the primes $p<N^{1/n}$, we can also write down explicitly the error term by applying Lemma \ref{l1}. By considering the number of possibilities for the coefficient of a polynomial $g\in\mathbb{F}_p[X]$ with a double root, a simple argument shows that there are $p^{n-1}$ such monic, degree-$n$ polynomials. Therefore,\begin{align*}
		\p_N(f\in\mathscr{P}_{n,N}^{0}:p|d_f)&=\frac{1}{|\mathscr{P}_{n,N}^{0}|}\underset{g\tiny\mbox{ double root}}{\underset{\tiny\mbox{monic, }\deg g=n}{\sum_{g\in\mathbb{F}_{p}[X]}}}\ \underset{f\equiv g\tiny\mbox{ mod }p}{\sum_{f\in \mathscr{P}_{n,N}^{0}}}1\\
		&=p^{n-1}\left( \frac{1}{p^n}+O(N^{-1})\right) \\
		&=\frac{1}{p}+O(p^{n-1}N^{-1}),
	\end{align*}as long as $p<N^{1/n}$. It follows that\begin{align*}
		\mathbb{E}_N(|\{p:p|d_f\}|)&=\sum_{ p<N^{1/n}}\frac{1}{|\mathscr{P}_{n,N}^{0}|}\underset{p|d_f}{\sum_{f\in \mathscr{P}_{n,N}^{0}}}1+O\Big(\frac{1}{|\mathscr{P}_{n,N}^{0}|}\sum_{f\in\mathscr{P}_{n,N}^0}\underset{p|d_f}{\sum_{ p\ge N^{1/n}}}1\Big)\\
		&=\log\log N+O_{n}(1).
	\end{align*}
	From (4), one has the claim.
\end{proof}

For a prime $p$, the ring $\Z[\alpha]$ is called $p$\textbf{-maximal} if $p$ is not a divisor of the index of $\Z[\alpha]$ in $\mathcal{O}_{K_1}$. In particular $\Z[\alpha]$ is not $p$-maximal if and only if $p|(d_f/D_{\Q_f})$.\\
There is an equivalent condition for $\Z[\alpha]$ to be $p$-maximal.

\begin{namedthm}{Theorem}[\cite{ABZ}, Corollary 3.2]
	The ring $\Z[\alpha]$ is not $p$-maximal if and only if there exists $u\in\Z[X]$, with $u$ mod $p$ irreducible, such that $f\in (p^2+up+u^2)$ in $\Z[X]$.
\end{namedthm}

In particular, the $p$-maximality depends just on $f$ mod $p^2$. The probability that such a polynomial modulo $p^2$ is in the above ideal (for a fixed $u$) is given by the following.

\begin{namedthm}{Theorem}[\cite{ABZ}, Proposition 3.4]
	Let $g\in\mathbb{F}_{p}[X]$ monic, of degree $m$; then$$\frac{1}{p^{2n}}\underset{f\in (gp+g^2) }{\underset{\tiny{\mbox{monic,\ }\deg f=n}}{\sum_{f\in(\Z/p^2\Z)[X]}}}1=\begin{cases}
		0&if\ 2m>n\\
		\frac{1}{p^{3m}}&if\ 2m\le n.
	\end{cases}$$
\end{namedthm}

\begin{coroll}In the above notations,
	the average of the number of primes dividing $a_f$ is$$
	\mathbb{E}_N(|\{p:p|a_f\}|)\ll_{n}1,$$as $N\rightarrow+\infty$.
\end{coroll}

\begin{proof}

	From theorems 2 and 3, we deduce that if $g\in\mathbb{F}_{p}[X]$ is monic, of degree $m\le n/2$, then
	\begin{align*}
		\frac{1}{|\mathscr{P}_{n,N}^{0}|}\underset{f\in (p^2+gp+g^2)}{\sum_{f\in\mathscr{P}_{n,N}^{0}}}1&=\frac{1}{|\mathscr{P}_{n,N}^{0}|}\underset{h\in (gp+g^2)}{\underset{\tiny{\mbox{monic,\ }\deg h=n}}{\sum_{h\tiny\mbox{ mod }p^2}}}\  \underset{f\equiv h\tiny\mbox{ mod }p}{\sum_{f\in \mathscr{P}_{n,N}^{0}}}1\\
		&=\underset{h\in (gp+g^2)}{\underset{\tiny{\mbox{monic,\ }\deg h=n}}{\sum_{h\tiny\mbox{ mod }p^2}}}\left( \frac{1}{p^{2n}}+O(N^{-1})\right) \\
		&=\frac{1}{p^{3m}}+O(p^{2n-3m}N^{-1}),
	\end{align*}for all primes $p<N^{1/2n}$.
	
	The next step is to compute the probability $\p_N(f\in \mathscr{P}_{n,N}^{0}:p|(d_f/D_{\Q_f}))$, which is, by the above\begin{multline*}
		\frac{1}{|\mathscr{P}_{n,N}^{0}|}\sum_{m\le n/2}\underset{\deg g=m}{\underset{\tiny{\mbox{monic, irreducible}}}{\sum_{g\in\mathbb{F}_{p}[X]}}}\ \underset{f\in(p^2+gp+g^2)}{\sum_{f\in \mathscr{P}_{n,N}^{0}}}1\\=\sum_{m\le n/2}\underset{\deg g=m}{\underset{\tiny{\mbox{monic, irreducible}}}{\sum_{g\in\mathbb{F}_{p}[X]}}}\left( \frac{1}{p^{3m}}+O(p^{2n-3m}N^{-1})\right) \\
		=\sum_{m\le n/2}\left( \frac{p^m}{m}+O\left( \frac{p^{m-1}}{m}\right) \right) \left( \frac{1}{p^{3m}}+O(p^{2n-3m}N^{-1})\right) \\
		=\sum_{m\le n/2}\left( \frac{1}{mp^{2m}}+O\left( \frac{p^{2n}}{mp^{2m}}N^{-1}+\frac{1}{mp^{2m+1}}\right) \right) \\
		=\frac{1}{p^2}+O\left( p^{2n-2}N^{-1}+\frac{1}{p^3}\right) ,
	\end{multline*}for $p<N^{1/2n}$. Then, by partial summation, the number of primes (on average) diving $d_f/D_{\Q_f}$ is\begin{multline*}
		\mathbb{E}_N(|\{p:p|(d_f/D_{\Q_f})\}|)\\
		=\sum_{p<N^{1/2n}}\p_N(f:p|(d_f/D_{\Q_f}))+\frac{1}{|\mathscr{P}_{n,N}^{0}|}\sum_{f\in \mathscr{P}_{n,N}^{0}}\underset{p|(d_f/D_{\Q_f})}{\sum_{p\ge N^{1/2n}}}1\\
		=\sum_{p<N^{1/2n}}\frac{1}{p^2}+O\left( 1+\frac{N^{1-\frac{1}{2n}}}{\log N}N^{-1}\right) \\
		\ll_{n} 1,
	\end{multline*}since$$\underset{p|(d_f/D_{\Q_f})}{\sum_{ p\ge N^{1/2n}}}1\ll_{n}\frac{\log N}{\log N}\ll_{n} 1.$$
\end{proof}

\subsection{Upper bounds for the torsion part of the class number}
All these bounds represent evidence towards the so-called $\varepsilon$-conjecture.
\begin{conj}
	Let $K/\Q$ be a number field of degree $s$ with discriminant $D_K$. Then for every integer $\ell\ge1$ and every $\varepsilon>0$,$$h_K[\ell]\ll_{s,\ell,\varepsilon}D_K^{\varepsilon},$$where $h_K[\ell]$ is the order of the $\ell$-torsion subgroup of the class group.
\end{conj}
Using the well-known Minkowski bound ($ r_2 $ is the number of real $ \Q $-embeddings of $ K $)$$h_K\le\frac{s!}{s^s}\frac{4^{r_2}}{\pi^{r_2}}D_K^{1/2}(\log D_K)^{s-1}$$one has$$h_K\ll_{s,\varepsilon} D^{1/2+\varepsilon}_K$$for any $ \varepsilon>0 $. We can of course use the above to bound the $ \ell $-part $ h_K[\ell] $ of $ h_K $. But we'd like to improve the above estimate, and the main point we're going to use is the existence of "many" splitting completely primes, which contributes significantly to the quotient of the class group by its $\ell$-torsion. See Theorem 4 for the precise statement. The GRH guarantees the existence of such primes, but here, we'd like to proceed unconditionally.\\

Let $K/\Q$ be a number field. The presence of "small" primes that split completely in $ K$, give a means to improve the Minkowski lower bound, using the following theorem (\cite{EV}, Lemma 2.3).
\begin{namedthm}{Theorem}[Ellenberg, Venkatesh]
	Let $ K/\Q $ be a field extension of degree $ s $, and let $\ell$ be a positive integer. Set $ \delta<\frac{1}{2\ell(s-1)} $ and suppose that$$|\{p\le D_K^\delta:p\mbox{ splits completely in }K/\Q\}|\ge M.$$Then, for any $ \varepsilon>0 $ $$h_K[\ell]\ll_{s,\ell,\varepsilon}D^{1/2+\varepsilon}_KM^{-1}. $$
\end{namedthm}
\noindent Back to our set, for each $ f \in\mathscr{P}_{n,N}^{0}$, let $h_f=h_{\Q_f}$.
\begin{coroll}\label{c5}
	For every positive integer $\ell$, $\varepsilon>0$ and for almost all $ f \in\mathscr{P}_{n,N}^{0}$, outside of a set of size $ o(N^{n}) $, we have$$h_f[\ell]\ll_{n,\ell,\varepsilon}D_{\Q_f}^{\frac{1}{2}-\frac{1}{(2n-2)\log\log d_f}+\varepsilon},$$as $ N\rightarrow+\infty $.
\end{coroll}

In order to prove Corollary \ref{c5}, we need the following lemma. If $G$ is a transitive subgroup of $S_n$, we denote by$$\mathscr{F}_s(X,G)=\{L/\Q:[L:\Q]=s,\ G_{\widetilde{L}/\Q}\cong G,\ D_K\le X\},$$where $\widetilde{L}$ is the Galois closure of $L$ over $\Q$. Moreover, in the following $H$ is the \textit{multiplicative Weil height}. If $\alpha$ is an algebraic number of degree $s$ and $f$ is its minimal polynomial over $\Q$, then the \textit{Mahler measure} of $f$ is$$M(f)=H(\alpha)^s.$$Mahler showed that $M(f)$ and $\h(f)$ are commensurate in the sense that$$\h(f)\ll M(f)\ll \h(f).$$In particular $H(\alpha)\ll N^{1/s}$.\\

We underline that the following result also holds if we put $D_f$ in place of $D_{\Q_f}$.

\begin{lemma}\label{l3}
		The density of the set of $f\in\mathscr{P}_{n,N}^{0}$ so that $D_{\Q_f}\le N^{1/\sqrt{\log\log N}}$ is zero, as $N\rightarrow+\infty$.
\end{lemma}
\begin{proof}
	We have that$$
		\underset{D_{\Q_f}\le N^{1/\sqrt{\log\log N}}}{\sum_{f\in\mathscr{P}_{n,N}^{0}}}1\\
		\ll\sum_{L\in\mathscr{F}_n(N^{1/\sqrt{\log\log N}},S_n)}\underset{K_f\cong \widetilde{L}}{\sum_{f\in \mathscr{P}_{n,N}^{0}}}1.$$For any $L$ as above with signature $(r_1,r_2)$,\begin{align*}
			\underset{K_f\cong \widetilde{L}}{\sum_{f\in \mathscr{P}_{n,N}^{0}}}1&\le|\{\alpha\in\mathcal{O}_L:\Q(\alpha)\cong L,\ H(\alpha)\ll_{n}N^{1/n}\}|\\
			&\ll_{n}|\Omega_{N^{1/n}}\cap(\mathcal{O}_L\setminus \Z)|,
		\end{align*}where for $Y\ge1$, $\Omega_Y$ is the subset of the Minkowski space $L_\infty=\R^{r_1}\times\Co^{r_2}$ of elements with Weil height over $\Q$ at most $Y$.\\
		By applying Davenport's Lemma \cite{Da}, which says that the volume of a region is approximated by the number of lattice points in it, and by computing the volume of $\Omega_Y$ we achieve$$|\Omega_Y\cap\Z^n|\ll_{n}Y^{n}(\log Y)^{r_1+r_2-1}.$$By Proposition 2.2 of \cite{LT},$$|\Omega_Y\cap\mathcal{O}_L|\ll_{n}Y^{n}(\log Y)^{r_1+r_2-1}$$as well. In particular, $$\underset{K_f\cong \widetilde{L}}{\sum_{f\in \mathscr{P}_{n,N}^{0}}}1\ll_{n}N(\log N)^{r_1+r_2-1}.$$It follows that$$\underset{D_{\Q_f}\ll N^{1/\sqrt{\log\log N}}}{\sum_{f\in\mathscr{P}_{n,N}^{0}}}1\ll N^{1+\frac{n+2}{4\sqrt{\log\log N}}+\varepsilon }$$for every $\varepsilon>0$, by using Schmidt bound \cite{Sc} for the number of field extensions with bounded discriminant. We conclude by observing that$$\frac{N^{1+\frac{n+2}{4\sqrt{\log\log N}}+\varepsilon }}{N^{n-1}}\underset{N\rightarrow+\infty}{\longrightarrow}0.$$
\end{proof}

\begin{proof}(Corollary \ref{c5})
By Theorem \ref{t1}, for $ x=N^{1/\log\log N} $, and $ \alpha=\alpha(N)=o(\log N) $ for large $ N $,
$$\p_N\left( -N^{1/\alpha}\le\frac{\pi_{f,r}(x)-\delta(r)\pi(x)}{(\delta(r)-\delta(r)^2)^{1/2}\pi(x)^{1/2}}1\le N^{1/\alpha}\right)\underset{N\rightarrow+\infty}{\longrightarrow}1.$$In particular,$$\pi_{f,r}(x)\ge\delta(r)\pi(x)-N^{1/\alpha}(\delta(r)-\delta(r)^2)^{1/2}\pi(x)^{1/2}$$for all but $ o(N^n) $ $ f $ in the set $ \mathscr{P}_{n,N}^{0} $. Pick $ \alpha=3\log\log N $; then $ N^{1/\alpha}\pi(x)^{1/2}\ll_r\pi(x) $. By enlarging $N$, we can assume that$$N^{1/\alpha}(\delta(r)-\delta(r)^2)^{1/2}\pi(x)^{1/2}\ll_r\frac{1}{2}\delta(r)\pi(x).$$Then we get\begin{equation}
		\pi_{f,r}(x)
	\gg\delta(r)\pi(x).
	\label{eq13}
\end{equation}

We consider the totally split primes (corresponding to the trivial conjugacy class).
 Since $d_f\asymp N^{2n-2}$ for almost all $f$, as stated in \ref{eq12}, and by the relation \ref{eq11} between discriminants, we are bounding from below the primes \begin{align*}
	p&\ll_n N^{1/\log\log N}\\
	&\ll(d_f^{1/(2n-2)})^{1/\log\log d_f}\\
	&=(D_{\Q_f}^{1/(2n-2)}a_f^{1/(n-1)})^{1/\log\log d_f}\\
	&=\left(D_{\Q_f}^{\frac{1}{(2n-2)}+\frac{\log a_f}{(n-1)\log D_{\Q_f}}}\right)^{1/\log\log d_f}
\end{align*} splitting completely in $ K_f/\Q $. 
By the relation between discriminants, one has$$a_f\ll N^{n-1},$$By Lemma \ref{l3}, for almost all $f$ the second summand in the exponent $\frac{\log a_f}{\log D_{\Q_f}\log\log d_f}$ goes to zero as $N\rightarrow+\infty$. Hence, we have that for almost all $f$, the exponent\begin{align*}
	\left(\frac{1}{(2n-2)}+\frac{\log a_f}{(n-1)\log D_{\Q_f}}\right)
	\cdot\frac{1}{\log\log d_f} \underset{N\rightarrow+\infty}{\longrightarrow}0.
\end{align*}
In particular, we can take $\delta>0$ so that
\begin{align*}
	\left(\frac{1}{(2n-2)}+\frac{\log a_f}{(n-1)\log D_{\Q_f}}\right)
	\cdot\frac{1}{\log\log d_f}<\delta<\frac{1}{2\ell(n!-1)}.
\end{align*}

It turns out, by \ref{eq13}, that the primes $ p\ll D_f^\delta $ splitting completely in $ K_f/\Q $ number at least$$\gg_{n,\ell}\frac{(d_f)^{\frac{1}{(2n-2)\log\log d_f}}\log\log d_f}{\log d_f}.$$Now, the number of primes splitting completely in $\Q_f/\Q$ is larger than the number of primes splitting completely in $ K_f/\Q $; therefore, by Theorem 4 together with (4) again,\begin{align*}
	h_f[\ell]&\ll_{n,\ell,\varepsilon}\frac{D_{\Q_f}^{\frac{1}{2}+\varepsilon}\log d_f}{D_{\Q_f}^{\frac{1}{(2n-2)\log\log d_f}}a_f^{\frac{1}{(n-1)\log\log d_f}}\log\log d_f}.	
\end{align*}In particular, since $(\log d_f)/a_f\ll\log D_{\Q_f}$,$$h_f[\ell]\ll_{n,\ell,\varepsilon}D_{\Q_f}^{\frac{1}{2}-\frac{1}{(2n-2)\log\log d_f}+\varepsilon}$$for any $ \varepsilon>0 $, for almost all $ f $ as $ N\rightarrow+\infty $.
\end{proof}

We can slightly improve this last upper bound by adding an additional hypothesis.
\begin{thm}\label{t5}
	Let $K$ be a number field of degree $s$. There exists $\theta\in\mathcal{O}_K-\Z$ whose minimal polynomial $f_{\theta}$ has height$$\mbox{ht}(f_{\theta})\le 3^{s}\Big(\frac{D_K}{s}\Big)^{\frac{s}{2s-2}}.$$
\end{thm}
\begin{proof}
	See \cite{GJ}, Appendix A. 
\end{proof}

\begin{coroll}\label{c6}Assume that for almost all $f\in\mathscr{P}_{n,N}^0$ outside of a set of size $o(N^n)$, $\Q_f$ is generated over $\Q$ by an element $ \theta$ of small height as in Theorem \ref{t5}. Then for every positive integer $\ell$, for every $\varepsilon>0$, and for almost all $f\in\mathscr{P}_{n,N}^0$ outside of a set of size $o(N^n)$, we have
	$$h_{f}[\ell]\ll_{n,\ell,\varepsilon} D_{\Q_f}^{\frac{1}{2}-\frac{1}{n(2n-2)\log\log d_f}+\varepsilon},$$
	as $N\rightarrow+\infty$.
\end{coroll} 
\begin{proof}In particular we have $\mbox{ht}(f_{\theta})\ll_n D_{\Q_f}^{\frac{n}{2n-2}}$.
	By \cite{GZ}, Corollary 2.2, it follows that for every $\varepsilon>0$, there exists $\delta=\delta(\theta,\varepsilon)$ such that, for $D_{\Q_f}$ large enough,$$\mathbb{P}(|d_{f_\theta}|>\delta D_{\Q_f}^n)>1-\varepsilon.$$We can then say that 
	 $d_{f_{\theta}}=D_{\Q_f}\cdot a^{2}_{f_{\theta}}\asymp D_{\Q_f}^{n}$ with high probability.

 Since $d_f\asymp N^{2n-2}$ for almost all $f$, and$$D_{\Q_f}\asymp \frac{N^{(2n-2)}}{a_f^2}$$for almost all $f$, we obtain$$D_{\Q_f}\asymp\frac{N^{n(2n-2)}}{a^{2}_{f_{\theta}}a_f^{2n}}.$$Rearranging gives$$N\asymp C_n(f,\theta)\cdot D_{\Q_f}^{\frac{1}{n(2n-2)}}$$for almost all $f\in\mathscr{P}_{n,N}^{0}$. We denoted by $$C_n(f,\theta)=(a_{f_{\theta}}a_f^{n})^{\frac{2}{n(2n-2)}}.$$

As in Corollary \ref{c5} we count the primes\begin{align*}
	p&\ll_n N^{1/\log\log N}\\
	&\ll \left( C_n(f,\theta)\cdot D_{\Q_f}^{\frac{1}{n(2n-2)}}\right) ^{1/\log\log d_f}\\
	&=\left( D_{\Q_f}^{\frac{1}{n(2n-2)}+\frac{\log C_n(f,\theta)}{\log D_{\Q_f}}}\right) ^{1/\log\log d_f}
\end{align*}splitting completely in $K_f/\Q$. 
Since $a_f\ll N^{n-1}$, and by the above $a_{f_{\theta}}^2\ll D_{\Q_f}^{n/2}$, $$C_n(f,\theta)\ll_n N\cdot D_{\Q_f}^{\frac{1}{2n-2}}.$$
By applying Lemma \ref{l3}, we see that the exponent $$\left( \frac{1}{n(2n-2)}+\frac{\log C_n(f,\theta)}{\log D_{\Q_f}}\right) \cdot\frac{1}{\log\log d_f}$$goes to 0 as $N\rightarrow+\infty$. We can then pick a$$\left( \frac{1}{n(2n-2)}+\frac{\log C_n(f,\theta)}{\log D_{\Q_f}}\right) \cdot\frac{1}{\log\log d_f}<\delta<\frac{1}{2\ell(n!-1)}.$$
The number of primes $p\ll D_{\Q_f}^{\delta}$ splitting completely is bounded above by$$\gg_{n,\ell}\frac{\Big(C_n(f,\theta)\cdot D_{\Q_f}^{\frac{1}{n(2n-2)}}\Big)^{1/\log\log d_f}\log\log d_f}{\log d_f},$$By Theorem 4 one gets$$h_{f}[\ell]\ll_{n,\ell,\varepsilon} D_{\Q_f}^{\frac{1}{2}-\frac{1}{n(2n-2)\log\log d_f}+\varepsilon}$$
as $N\rightarrow+\infty$, for every $\varepsilon>0$ and for almost all $f$. 
\end{proof}

	\section{Further results and problems}

\subsection{The range of $x$ and $N$}
 Regarding the range of $x,N$ for the average Chebotarev Theorem, in our proofs, the restriction $x\le N^{1/\log\log N}$ is essential, since the fact that the primes in our counting function are small enough with respect to $N$, allows to have a sufficient number of $S_n$-polynomials congruent to a given one modulo these primes. This is the statement of the crucial result Lemma \ref{l1}. It would be interesting to know in what interval of $x$ and $N$ these results actually hold. The latter is still an open problem.

\subsection{Other Galois groups}	
 Another goal is to provide similar results for polynomials in having as Galois group over $\Q$, either $S_n$ or a transitive proper subgroup of $S_n$. It would be interesting to exploit the Hilbert Irreducibility Theorem to get results for some group $G\subseteq S_n$.
 
 In general, for a subgroup $G\subseteq S_n$, the elements in a conjugacy class in $G$ necessarily have the same cycle type, but the converse need not to be true. That is, the cycle type of a conjugacy class in $G$ need not determine it uniquely. This uniqueness property does hold for cycle types for the full symmetric group, which implies that the cycle type of an $S_n$-polynomial having a square-free factorization mod $p$ uniquely determines the Frobenius element for an $S_n$-number field obtained by adjoining one root of it.\\
 Let's consider the case of the alternating group $A_n\subseteq S_n$. A single conjugacy class in $S_n$ that is contained in $A_n$ may split into two distinct classes. Also, note that the fact that conjugacy in $S_n$ is determined by cycle type, means that if $\sigma\in A_n$, then all of its conjugates in $S_n$ also lie in $A_n$. There is a full characterization of the behaviour of conjugacy classes in $A_n$.
 \begin{lemma}\label{l4}
 	A conjugacy class in $S_n$ splits into two distinct conjugacy classes under the action of $A_n$ if and only if its cycle type consists of distinct odd integers. Otherwise, it remains a single conjugacy class in $A_n$.
 \end{lemma}
 \begin{proof}
 	Note that the conjugacy class in $S_n$ of an element $\sigma\in A_n$ splits, if and only if there is no element $\tau\in S_n\setminus A_n$ commuting with $\sigma$. For if there is one, for each $\tau'\in S_n\setminus A_n$ we have$$\tau'\sigma\tau'^{-1}=\tau'\sigma\tau\tau^{-1}\tau'^{-1}=(\tau'\tau)\sigma(\tau'\tau)^{-1},$$and $\tau\tau'\in A_n$. On the other hand, if $\tau\sigma\tau^{-1}$ and $\sigma$, with $\tau\in S_n\setminus A_n$, are conjugated in $A_n$, then for some $\tau'\in A_n$, we have $\tau\sigma\tau^{-1}=\tau'\sigma\tau'^{-1}$, giving$$\tau'^{-1}\tau\sigma=\sigma\tau'^{-1}\tau,$$and hence $\tau'^{-1}\tau\in S_n\setminus A_n$ commutes with $\sigma$.
 	
 	Now suppose, $\sigma$ has a cycle $c_i$ of even length. A cycle of even length is an element of $S_n\setminus A_n$, and as $\sigma$ commutes with its cycles, we are done by the above. If $\sigma$ has two cycles $(a_1\dots a_k)$ and $(b_1\dots b_k)$ of the same odd length $k$, then $(a_1 b_1)\dots(a_kb_k)$ is a product of $k$ permutations (hence odd, so an element of $S_n\setminus A_n$) commuting with $\sigma$.
 	
 	Suppose $\sigma=c_1\dots c_s$ is a product of odd cycles $c_i$ of distinct lengths $d_i$. Let $\tau\in S_n$ be a permutation commuting with $\sigma$. Then $\tau$ must fix each of the $c_i$, that is, $\tau$ must be of the form $\tau=c_1^{a_1}\dots c_s^{a_s}$ for some $a_i\in\Z$. But as the $c_i$ are even permutations (as cycles of odd length), we have $\tau\in A_n$. So no $\tau\in S_n\setminus A_n$ commutes with $\sigma$ and we have the claim.
 \end{proof}

 As in the case of $S_n$ polynomials, we want to count the number of $G$-polynomials, where $G$ is a subgroup of the symmetric group $S_n$.
 We use the following result.

 \begin{namedthm}{Theorem}[Dietmann]
 	For every $\varepsilon>0$ and positive integer $n$,\begin{multline*}
 		|\{(a_0,\dots,a_{n-1})\in\Z^n:|a_j|\le N\ \forall j,\\
 		f(X)=X^n+a_{n-1}X^{n-1}+\dots+a_0\mbox{ has }G_f=G\}|\\ \ll_{n,\varepsilon}N^{n-1+1/[S_n:G]+\varepsilon},
 	\end{multline*}where $[S_n:G]$ is the index of $G$ in $S_n$.
 \end{namedthm}
 \begin{proof}
 See \cite{Di} and \cite{Di2}.
 \end{proof}

Consider the set$$\mathscr{P}_{n,N}^1=\{f\in\mathscr{P}_{n,N}:G_f=S_n\mbox{ or }A_n\}.$$

Now, if $G\subseteq S_n$, $G\neq S_n,A_n$ then its index in $S_n$ is greater or equal then $n$. From Theorem 6 we thus have that$$|\mathscr{P}_{n,N}^1|=(2N)^{n}+O( N^{n-1+\frac{1}{n} +\varepsilon }) .$$

Let $\mathscr{C}_r$ be a conjugacy class in $A_n$, with $r=(r_1,\dots,r_n)$ a square-free splitting type such that either $r_i$ is even for some $i$, or all the $r_i$'s are odd but $r_i=r_j$ for some $i\neq j$. 

By following the same argument as in Lemma \ref{l1} and Proposition \ref{p1}, we get the Chebotarev Theorem on average:$$\frac{1}{|\mathscr{P}_{n,N}^1|}\sum_{f\in \mathscr{P}_{n,N}^{1}}\Big(\underset{\Frob_{f,p}=\mathscr{C}_r}{\sum_{ p\le x}}1 \Big)=\delta(r)\pi(x)+C_r\log\log x+O_{n}(1),$$if $x<N^{\frac{1-(1/n)-\varepsilon}{n+1}}$.\\

\noindent\textbf{Remark.} Note that even if one doesn't fully control the conjugacy classes of a subgroup $G\subseteq S_n$ in terms of the cycle type, there is still interesting information to extract from it. Especially, about the number of totally splitting primes (corresponding to the trivial conjugacy class), which was the main tool in the application to class group torsion upper bounds of Section 5.2.	

\subsection{Generalization for number fields}

 One can consider the analogues over number fields, that is polynomials with coefficients in the ring of algebraic integers of a fixed finite extension $K/\Q$ of degree $d$. It is possible to generalize the work of Bhargava, towards the van der Waerden's conjecture. In this case, for a polynomial$$f(X)=X^n+\alpha_{n-1}X^{n-1}+\dots+\alpha_0\in\mathcal{O}_K[X],$$the \textit{height} is defined in terms of the integer coefficients of the $\alpha_i$ in a fixed integral basis of $\mathcal{O}_K$ over $\Z$. In particular, we get that the number of non-$S_n$-polynomials as above is$$\ll_{n,K}N^{d(n-1/2)}\log N,$$for all $d\ge1$, $n\ge2$, as $N\rightarrow+\infty$. The above result, as well as a refined version for some values of $d$ and $n$, are contained in my Ph.D. thesis as well as in  	
 https://doi.org/10.48550/arXiv.2212.11608.

\noindent
\footnotesize
DEPARTEMENT MATHEMATIK, ETH EIDGEN\"{O}SSICHE TECHNISCHE HOCHSCHULE, 8092 Z\"{U}RICH, SWITZERLAND\\
\textit{Email address}: \texttt{ilaria.viglino@yahoo.it}

\end{document}